\documentclass[11pt]{article}
\usepackage{amsmath,amssymb,amsthm,amsfonts}

\usepackage{url}

\usepackage[pdftex]{color}

\newcommand{\Rd}{{\mathbf{R}^d}}
\newcommand{\Rt}{{\mathbf{R}^2}}

\newtheorem{thm}{Theorem}[section]

\newtheorem{lem}[thm]{Lemma}
\newtheorem{prop}[thm]{Proposition}
\newtheorem{cor}[thm]{Corollary}
\newtheorem{remark}[thm]{Remark}

\theoremstyle{definition}
\newtheorem{defn}{Definition}[section]

\numberwithin{equation}{section}

\DeclareMathOperator{\dist}{dist}

\theoremstyle{definition}
 
\begin{document}

\begin{titlepage}
\title{\bf Trace Estimates  for Stable Processes}
\author{Rodrigo Ba\~nuelos\thanks{Supported in part by NSF Grant
\#0603701-DMS.}\\Department of  Mathematics
\\Purdue University\\West Lafayette, IN
47906\\banuelos@math.purdue.edu\and Tadeusz  
Kulczycki\thanks{Supported in part by KBN Grant 1 P03A 020 28}
\\Institute of Mathematics \\ Wroc{\l}aw
University of Technology\\ 50-370 Wroc{\l}aw, Poland\\
Tadeusz.Kulczycki@pwr.wroc.pl}
\maketitle

\begin{abstract}

In this paper we study the behaviour in time of the trace
(the partition function) of the heat semigroup associated
with symmetric stable processes in domains of $\Rd$. In
particular, we show that for domains with the so called
{\it{$R$-smoothness}} property the second terms in the asymptotic
as $t\to 0$ involves the surface area of the domain, just as
in the case of Brownian motion. 

\bigskip

\centerline{\bf  Contents}

\begin{itemize}
\item[{}]
\item[\S1.] {\sl Introduction and statement of main result}
\item[\S2.] {\sl Preliminaries}
\item[\S3.] {\sl Proof of main result}
\end{itemize}
\end{abstract}
\end{titlepage}

\section{Introduction and statement of main result} \label{sec:Introduction}

Let $X_t$ be a symmetric $\alpha$-stable process in $\Rd$, $\alpha \in
(0,2]$. This is a process with independent and stationary increments
and characteristic function 
$E^0 e^{i \xi X_t} = e^{-t |\xi|^{\alpha}}$, $\xi \in \Rd$, $ t > 0$. By $p(t,x,y) = p_{t}(x-y)$ we
will denote the transition density of this process starting at the point $x$. That is, 
$$
P^{x}(X_t \in B)= \int_{B} p(t,x,y)\, dy.
$$
Since the transition density is obtained from the characteristic function by the 
 inverse Fourier transform, it follows trivially that $p_t(x)$ is a 
radial symmetric decreasing function  and that 
\begin{equation}\label{kernelbound1}
p_t(x) = t^{-d/\alpha} p_1(t^{-1/\alpha} x) \le t^{-d/\alpha} p_1(0), \quad t > 0, \, x \in \Rd.
\end{equation}
Thus in fact
\begin{eqnarray}\label{constant1}
p_t(0) = t^{-d/\alpha} p_1(0)&=& t^{-d/\alpha}\, 
\frac{1}{(2 \pi)^d} \int_{\Rd} e^{-|x|^{\alpha}} \, dx\nonumber\\
& =& 
 t^{-d/\alpha}\, \frac{\omega_d}{(2 \pi)^d\alpha}
 \int_0^{\infty}e^{-s} s^{(\frac{n}{\alpha}  -1)} ds\nonumber\\
 &=& t^{-d/\alpha}\, \frac{\omega_d \Gamma(d/\alpha)}{(2\pi)^d\alpha},
\end{eqnarray}
where $\omega_d$ is the surface area of the unit sphere in $\Rd$.  Of course,  when $\alpha=2$, 
$p_t(0)=(4\pi t)^{-d/2}$, since $\omega_d=\frac{2\pi^{d/2}}{\Gamma(d/2)}$. 

In this paper we will be interested in the process $X_t$ in open sets of $\Rd$ and the 
behavior of the corresponding semigroup.  Let $D \subset \Rd$ be an open set and denote by $\tau_{D} = \inf\{t \ge 0: X_t \notin D\}$ the first exit time of $X_t$ from $D$. By $\{P_t^{D}\}_{t \ge 0}$ we denote the semigroup
on $L^2(D)$ of $X_t$ killed upon exiting $D$.  That is, for any $t>0$ and $f\in L^2(D)$ we define 
$$
P_{t}^{D}f(x) = E^{x}(\tau_{D} > t; f(X_{t})), \quad x \in D. 
$$
The semigroup has transition density $p_D(t,x,y)$ satisfying 
$$
P_{t}^{D}f(x) = \int_{D} p_{D}(t,x,y) f(y) \, dy
$$
and just as in the case of Brownian motion (case $\alpha=2$), 
\begin{equation}\label{PDformula}
p_D(t,x,y) = p(t,x,y) - r_D(t,x,y),
\end{equation}
where
\begin{equation}
\label{rD}
r_D(t,x,y) = E^x(\tau_D < t; p(t - \tau_D, X(\tau_D), y)).
\end{equation}
Whenever  $D$ is bounded (or of finite volume), the operator $P_t^D$ maps $L^2(D)$
into $L^{\infty}(D)$ for every $t>0$.  This follows from (\ref{kernelbound1}), (\ref{rD}), and the general theory of heat semigroups as described in \cite{Da}.  In fact, it follows from 
\cite{Da} that there exists an orthonormal basis of eigenfunctions
$\{\varphi_n\}_{n =1}^{\infty}$ for $L^2(D)$ and corresponding eigenvalues $\{\lambda_n\}_{n = 1}^{\infty}$of the generator of the semigroup $\{P_t^{D}\}_{t \ge 0}$ satisfying
$$0<\lambda_1<\lambda_2\leq \lambda_3\leq \dots$$
  with $\lambda_n\to\infty$
as
$n\to\infty$. That is,  the pair $\{\varphi_n, \lambda_n\}$ satisfies
\begin{equation*}
P_{t}^{D}\varphi_{n}(x) = e^{-\lambda_{n} t} \varphi_{n}(x), \quad x \in D, \,\,\, t > 0.
\end{equation*}
Under such assumptions we have 
\begin{equation}\label{eigenexpan}
p_D(t,x,y) = \sum_{n = 1}^{\infty} e^{-\lambda_{n} t} \varphi_{n}(x) \varphi_n(y).
\end{equation}

Let us point out that the generator of the semigroup $\{P_t^{D}\}_{t \ge 0}$ is the pseudodifferential operator
\begin{equation*} 
  -(-\Delta)^{\alpha / 2} f(x) =
 \lim_{\varepsilon \rightarrow 0^+} 
    {\cal A}_{d,-\alpha} \int\limits_{|y-x| > \varepsilon}
    \frac{f(y) - f(x)}{|x - y|^{d + \alpha}} \, dy\,,
\end{equation*}
where ${\cal A}_{d,\gamma}=\Gamma((d-\gamma)/2)/(2^{\gamma}\pi^{d/2}|\Gamma(\gamma/2)|)$, see \cite{CS1}.

The study of the ``fine" spectral theoretic properties of the killed semigroup of stable processes in domains of Euclidean space has been the subject of many papers in recent years, see for example,  \cite{CS2}, \cite{BLM}, \cite{M}, \cite{BK1}, \cite{BK2}, \cite{BK3}, \cite{Db}, \cite{DbM}, \cite{CS0}, \cite{CS3}, \cite{DK}, \cite{Kw}. In this paper we are interested in the behavior of the trace of this semigroup as $t\to 0$.  More precisely, we study the behavior as $t\to 0$ of the quantity 
\begin{equation}\label{trace1}
Z_D(t) = \int_{D} p_D(t,x,x) \, dx.
\end{equation}
Because of (\ref{eigenexpan}), we can re-write (\ref{trace1}) as
\begin{equation}\label{partition}
Z_D(t) = \sum_{n = 1}^{\infty} e^{-\lambda_{n} t} \int_{D} \varphi_{n}^2(x) \, dx = 
\sum_{n = 1}^{\infty} e^{-\lambda_{n} t}.
\end{equation}

The quantity $Z_D(t)$ is often referred to as the {\it partition function} of $D$. 
For any set $D\subset \Rd$ we denote its volume ($d$-dimensional Lebesgue measure) by $|D|$. It is shown in \cite{BG} that  
for any open set $D\subset \Rd$ of finite volume whose boundary,  $\partial D$, has zero $d$-dimensional Lebesgue measure, 
\begin{equation}\label{bg1}
Z_D(t) \sim\frac{C_1 |D|}{t^{d/\alpha}}, \,\,\, \text{as}\,\,\, t\to 0, 
\end{equation}
with 
$
C_1 =\frac{\omega_d \Gamma(d/\alpha)}{(2\pi)^d\alpha}.
$  
By 
(\ref{bg1}) we means that   
\begin{equation}\label{bg2}
\lim_{t\to 0}t^{d/\alpha}Z_D(t)=C_1 |D|.
\end{equation}
If we now let $N(\lambda)$ be the number of eigenvalues $\{\lambda_j\}$ which do not exceed $\lambda$, it follows from (\ref{bg1}) and the classical Karamata tauberian theorem (see for example \cite{F} or \cite{Simon}, p. 108) that 
\begin{equation}\label{bg3}
N(\lambda) \sim \frac{{C_1 |D|}}{\Gamma(d/\alpha+1)}\,\lambda^{d/\alpha},\,\,\, \text{as}\,\, \lambda\to\infty.
\end{equation}
This is the analogue for stable processes of the celebrated Weyl's asymptotic formula for the eigenvalues  of  the Laplacian.   As we shall show below, (\ref{bg2}) follows easily from (\ref{PDformula}) and (\ref{rD}). 

Our goal in this paper is to obtain the second term in the asymptotics of $Z_D(t)$ under some additional assumptions on the smoothness of $D$.  
Our result is inspired by a similar result  for Brownian motion by M. van den Berg, (\cite{B}, Theorem 1).  
To state it precisely we need a definition. 

\begin{defn}
\label{Rsmooth}
The  boundary, $\partial D$, of an open set $D$ in $\Rd$ is said to be {\it{$R$-smooth}} if for each point $x_0\in \partial D$ there are two open balls $B_1$ and $B_2$ with radii $R$ such that $B_1 \subset D$, $B_2 \subset \Rd \setminus (D \cup \partial D)$  and $\partial B_1 \cap \partial B_2 = x_0$.
\end{defn}

\begin{thm}
\label{mainthm}
Let $D \subset \Rd$, $d \ge 2$, be an open 
bounded set with $R$-smooth boundary. Let $|D|$ denote the volume ($d$-dimensional Lebesgue measure) of $D$ and $|\partial D|$ denote its surface area ($(d-1)$-dimensional Lebesgue measure) of its boundary. Suppose $\alpha\in (0, 2)$. Then 
\begin{equation}
\label{main}
\left|Z_D(t) - \frac{C_1 |D|}{t^{d/\alpha}}  
+ \frac{C_2 |\partial D| t^{1/\alpha}}{t^{d/\alpha}}\right|
\le \frac{C_3 |D| t^{2/\alpha}}{R^2 t^{d/\alpha}},
\quad t > 0,
\end{equation}
where
$$
C_1 = p_1(0) =\frac{\omega_d \Gamma(d/\alpha)}{(2\pi)^d\alpha},
$$

$$
C_2 = C_2(d,\alpha) = \int_{0}^{\infty} r_H(1,(x_1,0,\ldots,0),(x_1,0,\ldots,0)) \, dx_1,
$$

$C_3 = C_3(d,\alpha)$, $H = \{(x_1,\ldots,x_d) \in \Rd: \, x_1 > 0\}$ and $r_H$ is given by (\ref{rD}).  
\end{thm}

The asymptotic for the trace of the heat kernel when $\alpha=2$ (the case of the Laplacian 
with Dirichlet boundary condition in a  domain of $\Rd$), have been extensively studies by many authors.  The van den Berg \cite{B} result which inspired our result above 
states that under the $R$--smoothness condition when $\alpha=2$, 
\begin{equation}
\label{vanden}
\left|Z_D(t) - (4\pi t)^{-d/2}\left(|D|-\frac{\sqrt{\pi t}}{2}|\partial D|\right)\right|
\le \frac{C_d |D| t^{1-d/2}}{R^2}, \,\,\,  t>0.
\end{equation}
For domains with $C^1$ boundaries the result 
\begin{equation}
\label{brossardcarmona}
Z_D(t) = (4\pi t)^{-d/2}\left(|D|-\frac{\sqrt{\pi t}}{2}|\partial D|+o(t^{1/2})\right),\,\,\, t\to 0,
\end{equation}
 was proved by Brossard and Carmona in \cite{BroCar}.  R. Brown subsequently extended (\ref{brossardcarmona}) to Lipschitz domains  in \cite{russbrown}.  We refer the reader to \cite{B}, \cite{BroCar} and \cite{russbrown} for more on the literature and history of these type of asymptotic results as well as corresponding results for the counting function $N(\lambda)$.  It would be interesting to extend Brown's result to all $\alpha\in (0, 2)$ and we believe such a result is possible.  At present we do not see how to do this.  Finally, we should mention here that the emerging of the surface area of the boundary of $D$ is somewhat surprising in our setting since stable processes 
``do not see" the boundary.  That is, under our assumptions on $D$, for any $x\in D$, $P^x\{X_{\tau_D}\in \partial D\}=0$ (see \cite{Bo}, Lemma 6). What we were naively expecting for the second term was, perhaps, some quantity involving the L\'evy measure of the process. 

The paper is organized as follows.  In \S2 we present several preliminary results which will be used in \S3 for the proof of Theorem \ref{mainthm}.  
Throughout the paper we will use $c$ to denote positive constants that depend (unless otherwise explicitly stated) only on $d$ and $\alpha$ but whose value may change from line to line. 

\section{Preliminaries} \label{sec:Auxiliary}

We start by setting some standard notation and recalling some well known facts. The ball in $\Rd$ center at $x$ and radius $r$, $\{y \in \Rd: \, |x - y| < r\}$  will be denoted by 
$B(x,r)$ and we will use $\delta_D(x)$ to denote the distance from the point $x$ to the boundary, $\partial D$,  of $D$.  That is, $\delta_D(x)=\dist(x, \partial D)$. The L\'evy measure of the stable processes $X_t$ will be denoted by $\nu$.  Its density, which we will just write as $\nu(x)$, is given by 
\begin{equation}\label{density}
\nu(x) = \frac{{\mathcal{A}}_{d,-\alpha}}{|x|^{d + \alpha}},
\end{equation}
where ${\mathcal{A}}_{d,\gamma}=\Gamma((d-\gamma)/2)/(2^{\gamma}\pi^{d/2}|\Gamma(\gamma/2)|)$. 
We will need the following bound on the transition probabilities of the process $X_t$ which can be found in   \cite{Z}:  For all $x, y\in \Rd$ and $t>0$, 
\begin{equation}
\label{ptxy}
p(t,x,y) \le c \left(\frac{t}{|x - y|^{d + \alpha}} \wedge \frac{1}{t^{d/\alpha}}\right).
\end{equation}
Throughout the paper we will use the fact (\cite{Bo}, Lemma 6) that if $D \subset \Rd$ is an open bounded set satisfying a uniform outer cone condition, then $P^x(X(\tau_D) \in \partial D) = 0$ for any  $x \in D$. 
 The scaling properties of $p_t(x)$ are inherited by the kernels $p_D$ and $r_D$. Namely, 
$$
p_D(t,x,y) = \frac{1}{t^{d/\alpha}} \,  p_{D/t^{1/\alpha}}\left(1,\frac{x}{t^{1/\alpha}},\frac{y}{t^{1/\alpha}}\right),
$$
\begin{equation}
\label{scalingr}
r_D(t,x,y) = \frac{1}{t^{d/\alpha}} \, r_{D/t^{1/\alpha}}\left(1,\frac{x}{t^{1/\alpha}},\frac{y}{t^{1/\alpha}}\right).
\end{equation}
Also, both $p_D$ and $r_D$ are symmetric.  That is, $p_D(t,x,y) = p_D(t,y,x)$
 and $r_D(t,x,y) = r_D(t,y,x)$.    The Green function for the open set $D\subset \Rd$ will be denoted by 
 $G_D(x, y)$.  Recall that in fact, 
$$
G_D(x,y) = \int_0^{\infty} p_D(t,x,y) \, dt, \quad x, y \in \Rd
$$
and that for any such $D$ the expectation of the exit time of the processes
$X_t$ from $D$ is given by the integral of the Green function over the domain.  That is, 
$$E^x(\tau_D)=\int_{D} G_D(x,y) \, dy.$$

\begin{lem}
\label{rDestimate}
Let $D \subset \Rd$ be an open set. For any $x,y \in D$ we have
$$
r_D(t,x,y) \le c (\frac{t}{\delta_D^{d + \alpha}(x)} \wedge \frac{1}{t^{d/\alpha}}).
$$ 
\end{lem}
\begin{proof}
By (\ref{rD}) and (\ref{ptxy}) we see that 
\begin{eqnarray*}
r_D(t,x,y) &=&
E^y\left(\tau_D < t; \, p(t - \tau_D,X(\tau_D),x)\right) \\
&\le& c E^y\left(\frac{t}{|x - X(\tau_D)|^{d + \alpha}} \wedge \frac{1}{t^{d/\alpha}}\right) \\
&\le& c \left(\frac{t}{\delta_D^{d + \alpha}(x)} \wedge \frac{1}{t^{d/\alpha}}\right).
\end{eqnarray*}
\end{proof}

\begin{remark}
Before we proceed, let us observe how this estimate implies the Blumenthal--Getoor (\ref{bg2}) estimate given above.  Indeed, by (\ref{PDformula}) we see that 
\begin{equation}\label{bg4}
\frac{p_D(t,x,x)}{p(t,x,x)}=1 - \frac{r_D(t,x,x)}{p(t,x,x)}
\end{equation}
and since 
$$p(t,x,x)=\frac{C_1}{t^{d/\alpha}},$$  
we see that (\ref{bg4}) is equivalent to 
\begin{equation}\label{bg5}
\frac{t^{d/\alpha}}{C_1}\, p_D(t,x,x)=1 - \frac{t^{d/\alpha}}{C_1}\, r_D(t,x,x).
\end{equation}
Thus in order to prove (\ref{bg2}), we must show that 
\begin{equation}\label{bg6}
\frac{t^{d/\alpha}}{C_1}\,\int_D r_D(t,x,x)\, dx \to 0,  \,\,\, \text{as}\,\,\, t\to 0. 
\end{equation}

For $0<t<1$, consider the sub-domains  $D_t=\{x\in D: \delta_D(x)\geq t^{1/2\alpha}\}$ and its complement $D_t^c=\{x\in D: \delta_D(x)<t^{1/2\alpha}\}$.   Under the assumption that $|D|<\infty$ and that the $d$-dimensional Lebesgue measure of its boundary is zero,  we have that $|D_t^c|\to 0$, as $t\to 0$.  (As pointed out us by the referee, the characteristic function of the set $D_t^c$ tends to zero pointwise and since $D$ has finite volume, the Lebesgue dominated convergence thoerem implies that $|D_t^c|\to 0$ without the assumption made in \cite{BG} that $\partial D$  has zero $d$-dimensional Lebesgue measure.)  Since $p_D(t,x,x)\leq p(t, x, x)$, by (\ref{bg5}) we see that $$\frac{t^{d/\alpha}}{C_1}\, r_D(t,x,x)\leq 1,$$ for all $x\in D$.  It follows  that 
\begin{equation}\label{bg7}
\frac{t^{d/\alpha}}{C_1}\,\int_{D_t^c} r_D(t,x,x)\, dx \to 0,  \,\,\, \text{as}\,\,\, t\to 0. 
\end{equation}
On the other hand, by Lemma \ref{rDestimate} we have 
\begin{equation}\label{bg8}
\frac{t^{d/\alpha}}{C_1} r_D(t,x,x) \le c \left(\frac{t^{d/\alpha+1}}{\delta_D^{d + \alpha}(x)} \wedge 1\right).
\end{equation}
For $x\in D_t$ and $0<t<1$, the right hand side of (\ref{bg8}) is bounded above by $ct^{d/2\alpha +1/2}$ and therefore 
\begin{equation}\label{bg9}
\frac{t^{d/\alpha}}{C_1}\,\int_{D_t} r_D(t,x,x)\, dx \leq ct^{d/2\alpha +1/2}|D| 
\end{equation}
and this last quantity goes to $0$ as $t\to 0$.  This proves  (\ref{bg2}).
\end{remark}

\begin{prop}
\label{pFpD}
Let $D$ and  $F$ be open sets in $\Rd$ such that $D \subset F$. Then for any $x, y \in \Rd$ we have
\begin{equation*}
 p_F(t,x,y) - p_D(t,x,y) 
= E^x(\tau_D < t, X(\tau_D) \in F \setminus D; p_F(t - \tau_D, X(\tau_D),y)).
\end{equation*}
\end{prop}

\begin{proof}
We have
\begin{eqnarray}
 p_F(t,x,y) - p_D(t,x,y) 
&=& r_D(t,x,y) - r_F(t,x,y) \\
\label{pFpD1}
&=& E^x(\tau_D < t; p(t - \tau_D, X(\tau_D),y)) \\
\label{pFpD1a}
 &  & \, - E^x(\tau_F < t; p(t - \tau_F, X(\tau_F),y)).
\end{eqnarray}
Note that on the set $\tau_D = \tau_F$ both expected values are equal. We also have $\tau_D \le \tau_F$ so (\ref{pFpD1}--\ref{pFpD1a}) equal 
\begin{eqnarray}
\label{pFpD3}
&& E^x(\tau_D < t, \, \tau_D < \tau_F; p(t - \tau_D, X(\tau_D),y)) \\
\label{pFpD4}
&-& E^x(\tau_F < t, \, \tau_D < \tau_F; p(t - \tau_F, X(\tau_F),y)).
\end{eqnarray}
Now we will prove the key equality
\begin{eqnarray}
\label{pFpD5}
&& E^x(\tau_F < t, \, \tau_D < \tau_F; p(t - \tau_F, X(\tau_F),y)) \\
\label{pFpD6}
&=& E^x(\tau_D < t, \, \tau_D < \tau_F; r_F(t - \tau_D, X(\tau_D),y)).
\end{eqnarray}
First, conditioning we see that 
\begin{eqnarray*}
&& E^x(\tau_D < t, \, \tau_D < \tau_F; r_F(t - \tau_D, X(\tau_D),y)) \\
&=& E^x \left[\tau_D < t, \, \tau_D < \tau_F; E^{X(\tau_D)}(\tau_F < t - s; p(t - s - \tau_F, X(\tau_F), y))\left|_{s = \tau_D} \right.\right]
\end{eqnarray*}
By the strong Markov property this equals 
\begin{eqnarray*}
&& E^x \left[\tau_D < t, \, \tau_D < \tau_F,  
\right. \\
&& \quad \quad \quad \quad  \times \left.
\tau_F \circ \Theta_{\tau_D} + s < t; p(t - s - \tau_F \circ \Theta_{\tau_D}, X(\tau_F) \circ \Theta_{\tau_D}, y)\left|_{s = \tau_D} \right.\right] \\
&=& E^x \left[\tau_D < t, \, \tau_D < \tau_F,   \right. \\
&& \quad \quad \quad \quad   \times \left. 
\tau_F \circ \Theta_{\tau_D} + \tau_D < t; p(t - \tau_F \circ \Theta_{\tau_D} - \tau_D, X(\tau_F) \circ \Theta_{\tau_D}, y)\right].
\end{eqnarray*}
Note that on the set $\tau_D < \tau_F$ we have $\tau_F \circ \Theta_{\tau_D} + \tau_D = \tau_F$ and $X(\tau_F) \circ \Theta_{\tau_D} = X(\tau_F)$.  So the last expression equals  
$$
E^x \left[\tau_D < t, \, \tau_D < \tau_F, \, \tau_F < t; p(t - \tau_F , X(\tau_F), y)\right]
$$
which is the same as (\ref{pFpD5}). This proves the equalities (\ref{pFpD5} - \ref{pFpD6}). Note that the condition $\tau_D < \tau_F$ may be written as $X(\tau_D) \in F \setminus D$. Hence  (\ref{pFpD5} - \ref{pFpD6}) and (\ref{pFpD3} - \ref{pFpD4}) imply the assertion of the proposition.
\end{proof} 

We will need the following well known estimate on the 
Green function of the complement of the unit ball. This  
follows from \cite{CS1}, Lemma 2.5.
\begin{lem}
\label{Green}
Let $\Omega = \overline{B(w,1)}^c$, $w \in \Rd$, $d \ge 2$. We have
$$
G_{\Omega}(x,y) \le \frac{c |x - w|^{\alpha/2} \, \delta_{\Omega}^{\alpha/2}(y)}{|x - y|^{d - \alpha/2}},
\quad x, y \in \Omega.
$$
\end{lem}

We will say that an open set $D \subset \Rd$ satisfies {\it{the uniform
outer ball condition with radius $1$}} if at each point $z \in \partial
D$ there exists a ball $B(w,1) \subset D^c$ such that $\partial D \cap
\partial B(w,1) = z$.

An easy corollary  of Lemma \ref{Green} is the  following result.
\begin{cor}
\label{Green1}
Let $D \subset \Rd$, $d \ge 2$ be an 
open set satisfying the uniform outer ball property 
with radius $1$. Then we have
$$
G_{D}(x,y) \le \frac{c \delta_D^{\alpha/2}(y) \, (|x - y| + 
\delta_D(x) + 1)^{\alpha/2}}{|x - y|^{d - \alpha/2}},
\quad x, y \in D.
$$
\end{cor}
\begin{proof}
Let $y \in D$ and $y_* \in \partial D$ be such that $|y - y_*| = \delta_D(y)$. There exists a ball $B(w,1) \subset D^c$ such that $\partial D \cap \partial B(w,1) = y_*$.
By Lemma \ref{Green} we obtain that $G_D(x,y)$ is bounded from above by
$$
G_{\overline{B(w,1)}^c}(x,y) \le \frac{c |x - w|^{\alpha/2} \, \delta_{D}^{\alpha/2}(y)}{|x - y|^{d - \alpha/2}} \le
\frac{c \delta_D^{\alpha/2}(y) \, (|x - y| + \delta_D(x) + 1)^{\alpha/2}}{|x - y|^{d - \alpha/2}}.
$$
\end{proof}

\begin{lem}
\label{ring}
Let $d \ge 2$, $b > 0$, $\Omega = B(0,2 b) \setminus \overline{B(0,b)}$ and $x \in \Rd$. Then we have
$$
E^x(\tau_{\Omega}) \le c b^{\alpha/2} \delta_{\Omega}^{\alpha/2}(x),
$$
$$
P^x(X(\tau_{\Omega}) \in B^c(0,2 b)) \le c b^{- \alpha/2} \delta_{B(0,b)}^{\alpha/2}(x).
$$
\end{lem}
\begin{proof}
For any open set $D \subset \Rd$, 
Borel set $A \subset \Rd$, $b > 0$, $x \in \Rd$ 
we have the following scaling properties
$$
E^{bx}(\tau_{bD}) = b^{\alpha} E^{x}(\tau_{D}),
$$
$$
P^{bx}(X(\tau_{bD}) \in bA) = P^{x}(X(\tau_{D}) \in A).
$$
It follows that we only need to deal with the case  $b = 1$.

The ring $\Omega = B(0,2) \setminus \overline{B(0,1)}$ is a bounded $C^{1,1}$ domain.
For bounded $C^{1,1}$ domains it 
is known that  $E^x(\tau_D) \le c(D,\alpha)
 \delta_D^{\alpha/2}(x)$, (\cite{K}, Proposition 4.9)
and there are also well known estimates for  
 $P^{x}(X(\tau_{D}) \in \cdot)$ (see \cite{CS1}, 
 Theorem 1.5, see also \cite{CS2}, Theorem 1.2).
  The lemma for $b = 1$ follows from these estimates.
\end{proof}

\begin{lem}
\label{intrinsic}
Let $T > 0$, $d \ge 2$ and $\Omega = B(0,2) \setminus \overline{B(0,1)}$. 
There exists a constant $C_T$ 
(depending on $T$, $d$, $\alpha$) such that for any $t \ge T$ we have
$$
p_{\Omega}(t,x,y) \le C_T \delta_{\Omega}^{\alpha/2}(x) 
\delta_{\Omega}^{\alpha/2}(y).
$$
\end{lem}
\begin{proof}
It is well known (\cite{CS2}, Theorem 4.6) that the semigroup
 $\{P_t^{\Omega}\}_{t \ge 0}$ is intrinsically ultracontractive. It follows that for any $t \ge T > 0$ we have
$$
p_{\Omega}(t,x,y) \le C_T \varphi_1(x) \varphi_1(y),
$$
where $\varphi_1$ is the ground state eigenfunction for $\Omega$. It is also well known (\cite{CS2}, Theorem 4.2)  that $\varphi_1(x) \le
c \delta_{\Omega}^{\alpha/2}(x)$, and the lemma follows.
\end{proof}

We will need the following ``space-time" generalization of the Ikeda-Watanabe formula \cite{IW}. Such a generalization has been proved for the relativistic stable process in \cite{KS}, Proposition 2.7. The proof of this generalization in our case is exactly the same as in \cite{KS} and is omitted.
\begin{prop}
\label{IWprop}
Let $D$ be an open nonempty set and $A$  a Borel set such that $A \subset D^c \setminus \partial D$. Assume that $0 \le t_1 < t_2 < \infty$, $x \in D$. Then we have
$$
P^x(X(\tau_D) \in A, \, t_1 < \tau_D < t_2) = 
\int_D \int_{t_1}^{t_2} p_D(s,x,y) \, ds \int_A \nu(y - z) \, dz \, dy.
$$
\end{prop}

The following proposition is already known for relativistic stable process \cite{KS} (see Theorem 4.2).
\begin{prop}
\label{heat}
Let $\Omega = (\overline{B(w,1)})^c$, $w \in \Rd$, $d \ge 2$. There exists a constant $c$ such that for any $t > 0$, $x ,y \in \Omega$ and $|x - y| \ge a > 0$, we have
\begin{equation}
\label{heatineq}
p_{\Omega}(t,x,y) \le \frac{c (t \vee 1) \delta_{\Omega}^{\alpha/2}(y)}{(a \wedge 1)^{\alpha/2} |x - y|^{d + \alpha}}
\end{equation}
\end{prop}
\begin{proof}
 The proof of this proposition is very similar to the proof of Theorem 4.2 in \cite{KS}.
 We will assume that $w = 0$, so that $\Omega = (\overline{B(0,1)})^c$. We have
$$
p_{\Omega}(t,x,y) \le p(t,x,y) \le \frac{c t}{|x - y|^{d + \alpha}}.
$$
Thus  for $y$ such that $\delta_{\Omega}(y) \ge (a \wedge 1)/8$ the proposition holds trivially.
From now we suppose that  $\delta_{\Omega}(y) < (a \wedge 1)/8$. Let us also assume that $y = (|y|,0,\ldots,0)$. Consider the  ring $R = B(p,2 b) \setminus \overline{B(p,b)}$, where $p = (1 - b,0\ldots,0)$ and $b = (a \wedge 1)/8$. Note that $\delta_{\Omega}(y) = \delta_{B(p, b)}(y)$. 

In order to show (\ref{heatineq}) we will 
estimate the integral of $p_{\Omega}(t,z,y)$ over the 
smaller  ball $B(x,s)$, $s < b$.  We will then 
differentiate this quantity by dividing by the 
volume and taking the limit as $s$ tends to $0$. 
First observe that  $B(x,s) \subset R^c$. 
We have
\begin{eqnarray*}
\int_{B(x,s)} p_{\Omega}(t,z,y) \, dz &=&
P^y(X(t) \in B(x,s), \tau_D > t) \\
&\le& P^y(\tau_R < t, X(\tau_R) \in \Omega \setminus R, X(t) \in B(x,s)).
\end{eqnarray*}
By the strong Markov property the last expression equals
\begin{equation}
\label{tauR}
E^y\left[\tau_R < t, X(\tau_R) \in \Omega \setminus R; \, P^{X(\tau_R)}(X(t - r) \in B(x,s))|_{r = \tau_R}\right].
\end{equation}
Let $A = B(x,|x - y|/4)$. Note that $A \subset R^c$. We will divide the set $\Omega \setminus R$ into two subsets $A \cap \Omega$ and $F = \Omega \setminus (A \cup R)$. Observe that  
\begin{eqnarray*}
&& E^y\left[\tau_R < t, X(\tau_R) \in F; \, P^{X(\tau_R)}(X(t - r) \in B(x,s))|_{r = \tau_R}\right] \\
&=& E^y\left[\tau_R < t, X(\tau_R) \in F; \, \int_{B(x,s)} p(t - \tau_R, X(\tau_R),z) \, dz\right].
\end{eqnarray*}
Note also that $X(\tau_R) \in F$, so for $z \in B(x,s)$, $s < b \le |x - y|/8$ we have $|X(\tau_R) - z| \ge |x - y|/8$. By (\ref{ptxy}) this is bounded above by 
$$
c P^y(X(\tau_R) \in F) \frac{t|B(x,s)|}{|x - y|^{d + \alpha}}.
$$
By Lemma \ref{ring} and the fact that $\delta_{\Omega}(y) = \delta_{B(p,b)}(y)$ this is bounded above by
$$
\frac{c t \delta_{\Omega}^{\alpha/2}(y) |B(x,s)|}{b^{\alpha/2} |x - y|^{d + \alpha}}.
$$

Now let us estimate the part of (\ref{tauR}) corresponding 
to the set $A \cap \Omega$. By the ``space-time"
Ikeda-Watanabe formula stated above, (Proposition \ref{IWprop}), we have
\begin{eqnarray}
&& E^y\left[\tau_R < t, X(\tau_R) \in A \cap \Omega; \, P^{X(\tau_R)}(X(t - r) \in B(x,s))|_{r = \tau_R}\right]\nonumber \\
 &&=  P^y\left[\tau_R < t, X(\tau_R) \in A \cap \Omega, \, X(t) \in B(x,s) \right]\label{tauR1}\\
 && =  \int_R \int_{0}^{t} p_R(r,y,u)  \int_{A \cap \Omega} \nu(u - v) P^v(X(t - r) \in B(x,s)) \, dv \, dr \, du.\nonumber
\end{eqnarray}
Note that for $u \in R$, $v \in A \cap \Omega$, we have $|u - y| \le 4b \le |x - y|/2$, $|v - x| \le |x - y|/4$.
Thus 
$$
\nu(u - v) \le c |x - y|^{-d - \alpha}, \quad u \in R, \, v \in A \cap \Omega.
$$
We also have
$$
\int_{A \cap \Omega} P^v(X(t - r) \in B(x,s)) \, dv =
\int_{B(x,s)} \int_{A \cap \Omega} p(t - r, v,z) \, dv \, dz \le
|B(x,s)|
$$
and
\begin{eqnarray*}
\int_R \int_{0}^{t} p_R(r,y,u) \, dr \, du 
&\le& \int_R G_R(y,u) \, du = E^y(\tau_R) \\
&\le& c b^{\alpha/2} \delta_{B(p,b)}^{\alpha/2}(y) =
c b^{\alpha/2} \delta_{\Omega}^{\alpha/2}(y).
\end{eqnarray*}
It follows that (\ref{tauR1}) is bounded above by 
$$
\frac{c b^{\alpha/2} \delta_{\Omega}^{\alpha/2}(y) |B(x,s)|}{ |x - y|^{d + \alpha}}.
$$
Recall that $b = (a \wedge 1)/8$. Finally diving both sides by $ |B(x,s)|$ gives 
$$
\frac{1}{|B(x,s)|} \int_{B(x,s)} p_{\Omega}(t,z,y) \, dz \le
\frac{c (t \vee 1) \delta_{\Omega}^{\alpha/2}(y)}{b^{\alpha/2} |x - y|^{d + \alpha}}.
$$
Letting $s \to 0$ we get the assertion of the proposition.
\end{proof}

An immediate corollary of the above result is 
\begin{cor}
\label{heatcor}
Let $D \subset \Rd$, $d \ge 2$, be an open set satisfying the uniform outer ball condition of radius $1$. 
There exists a constant $c$ such that for any $t > 0$, $x ,y \in D$ with $|x - y| \ge a > 0$, we have
$$
p_{D}(t,x,y) \le \frac{c (t \vee 1) \delta_{D}^{\alpha/2}(y)}{(a \wedge 1)^{\alpha/2} |x - y|^{d + \alpha}}.
$$
\end{cor}

\begin{prop}
\label{heat1}
Let $\Omega = (\overline{B(w,1)})^c$, $w \in \Rd$, $d \ge 2$ and $0 < S < T < \infty$. There exists a constant $c_{S,T}$ (depending on $S$, $T$, $d$, $\alpha$) such that for any $t \in [S,T]$ we have
$$
p_{\Omega}(t,x,y) \le c_{S,T} \delta_{\Omega}^{\alpha/2}(y), \quad x,y \in \Omega.
$$
\end{prop}
\begin{proof}
We assume that $w = 0$. We have
$$
p_{\Omega}(t,x,y) \le p(t,x,y) \le c t^{-d/\alpha},
$$
so when $\delta_{\Omega}(y) \ge 1/2$ the proposition holds trivially. 
Thus  we may assume that $\delta_{\Omega}(y) < 1/2$. Let $R = B(0,2) \setminus \overline{B(0,1)}$. By Proposition \ref{pFpD} $p_{\Omega}(t,x,y)$, equals 
$$
p_{R}(t,x,y) + E^x(\tau_R < t, X(\tau_R) \in \Omega \setminus 
R \setminus D; p_{\Omega}(t - \tau_R, X(\tau_R),y)).
$$
By Lemma \ref{intrinsic} and Lemma \ref{ring} we obtain 
$$
p_{R}(t,x,y) \le c_S \delta_R^{\alpha/2}(y) = c_S \delta_{\Omega}^{\alpha/2}(y).
$$
Since $\delta_{\Omega}(y) < 1/2$ and $|X(\tau_R)| \ge 2$, we see that 
$|X(\tau_R) - y| \ge 1/2$. By Proposition \ref{heat} we obtain
$$
p_{\Omega}(t - \tau_R, X(\tau_R),y) \le c T \delta_{\Omega}^{\alpha/2}(y),
$$
and the proposition follows.
\end{proof}

\begin{cor}
\label{heat1cor}
Let $D \subset \Rd$, $d \ge 2$ be an open set satisfying the uniform outer ball property with radius $1$. Let $0 < S < T < \infty$. Then there exists a constant $c_{S,T}$ (depending on $S$, $T$, $d$, $\alpha$) such that for any $t \in [S,T]$ we have
$$
p_{D}(t,x,y) \le c_{S,T} \delta_{D}^{\alpha/2}(y), \quad x,y \in D.
$$ 
\end{cor}

We will need some facts concerning the ``stability" of surface area of the boundary open sets with $R$-smooth boundary under certain perturbations.  
The following lemma is proved by van den Berg in  \cite{B}.
\begin{lem}[\cite{B}, Lemma 5]
\label{Dq}
Let $D$ be an open bounded set in $\Rd$ with 
$R$-smooth boundary $\partial D$ and define for $0 \le q < R$
$$
D_q = \{x \in D: \, \delta_{D}(x) > q \}
$$
and denote the area of its boundary $\partial D_q$ by $|\partial D_q|$. Then
\begin{equation}
\label{Dqi}
\left(\frac{R - q}{R}\right)^{d - 1} |\partial D| \le
|\partial D_q| 
\le \left(\frac{R}{R - q}\right)^{d - 1} |\partial D|, 
\quad 0 \le q < R. 
\end{equation}
\end{lem}
This lemma is formulated in \cite{B} for open bounded regions but it follows easily that it holds for all open bounded sets. Using this lemma we obtain the following result.
\begin{cor}
\label{partialDq}
Let $D$ be an open bounded set in $\Rd$ with $R$-smooth boundary. For any $0 < q \le R/2$ we have
\begin{itemize}
\item[(i)]
$$
2^{-d + 1} |\partial D| \le |\partial D_q| \le 2^{d - 1} |\partial D|,
$$
\item[(ii)]
$$
|\partial D| \le \frac{2^d |D|}{R},
$$
\item[(iii)]
$$
\left| |\partial D_q| - |\partial D| \right|
\le \frac{2^{d} d q |\partial D|}{R} 
\le \frac{2^{2 d} d q |D|}{R^2}. 
$$
\end{itemize}
\end{cor}

\begin{proof}
(i) follows directly from (\ref{Dqi}) under our restriction on $q$. 
By (i) we obtain
$$
|D| \ge |D \setminus D_{R/2}| = \int_{0}^{R/2} |\partial D_q| \, dq 
\ge 2^{-d} |\partial D| R,
$$
which gives (ii). 

By (\ref{Dqi}) we get 
\begin{equation*}
\left(\left(\frac{R - q}{R}\right)^{d - 1} - 1 \right) |\partial D| \le
|\partial D_q| - |\partial D| 
\le  \left(\left(\frac{R}{R - q}\right)^{d - 1} - 1 \right) |\partial D|.
\end{equation*}
Now (iii) follows from the mean value theorem and the fact that the derivatives of both $(\frac{R}{R - q})^{d - 1}$ and $(\frac{R - q}{R})^{d - 1}$ with respect to $q \in (0,R/2]$ are bounded by $2^d d R^{-1} $.
\end{proof}

\section{Proof of main result} \label{sec:Proof}

\begin{proof}[Proof of Theorem \ref{mainthm}]

We begin by observing that for  $t^{1/\alpha} > R/2$, the theorem holds trivially. Indeed for such $t's$ we have
$$
Z_D(t) \le \int_D p(t,x,x) \, dx \le \frac{c |D|}{t^{d/\alpha}} \le 
 \frac{c |D| t^{2/\alpha}}{R^2 t^{d/\alpha}}.
$$
By Corollary \ref{partialDq} (ii) we also have
$$
\frac{C_2 |\partial D| t^{1/\alpha}}{t^{d/\alpha}} \le
\frac{2^d C_2 |D| t^{1/\alpha}}{R t^{d/\alpha}} \le
\frac{2^{d + 1} C_2 |D| t^{2/\alpha}}{R^2 t^{d/\alpha}}.
$$
Therefore for $t^{1/\alpha} > R/2$ (\ref{main}) follows.

From now on we shall assume that $t^{1/\alpha} \le R/2$. From (\ref{PDformula}) and the fact that 
$p(t, x, x)=\frac{1}{t^{d/\alpha}}p_1(0)$, we see that 
\begin{eqnarray}
\nonumber
Z_D(t) - \frac{C_1 |D|}{t^{d/\alpha}} &=&
\int_{D} p_D(t,x,x) \, dx - \int_{D} p(t,x,x) \, dx \\
\label{main1}
&=& - \int_{D} r_D(t,x,x) \, dx,
\end{eqnarray}
where $C_1=p_1(0)$ as stated in the theorem.  Therefore we must estimate (\ref{main1}). 
We will use the notation of Lemma \ref{Dq}. We break our domain into 
two pieces, $D_{R/2}$  
and its complement.   We will first deal with the contribution in $D_{R/2}$. 
\bigskip

\noindent {\bf Claim I:}
 
\begin{equation}
\label{main2}
\int_{D_{R/2}} r_D(t,x,x) \, dx \le \frac{c |D| t^{2/\alpha}}{R^2 t^{d/\alpha}}, 
\end{equation}
for $t^{1/\alpha} \le R/2$. 
To verify this, observe that by scaling the left hand side of (\ref{main2}) equals 
\begin{equation}
\label{D-DR2}
\frac{1}{t^{d/\alpha}} \int_{D_{R/2}} r_{D/t^{1/\alpha}} \left(1,\frac{x}{t^{1/\alpha}},\frac{x}{t^{1/\alpha}}\right) \, dx.
\end{equation}
For $x \in D_{R/2}$ we  have $\delta_{D/t^{1/\alpha}}(x/t^{1/\alpha}) \ge R/(2 t^{1/\alpha}) \ge 1$. It follows by Lemma \ref{rDestimate} that 
$$
r_{D/t^{1/\alpha}} \left(1,\frac{x}{t^{1/\alpha}},\frac{x}{t^{1/\alpha}}\right)
\le \frac{c}{\delta_{D/t^{1/\alpha}}^{d + \alpha}(x/t^{1/\alpha})}
\le \frac{c}{\delta_{D/t^{1/\alpha}}^{2}(x/t^{1/\alpha})}
\le \frac{c t^{2/\alpha}}{R^2}.
$$
Hence (\ref{D-DR2}) is bounded by $c|D|t^{2/\alpha}/(R^2t^{d/\alpha})$, which gives (\ref{main2}). 

 Now let us introduce the following notation. Since $D$ has $R$-smooth
boundary, for any point $y \in \partial D$ there are two open balls
$B_1$ and $B_2$ both of radius $R$ such that $B_1 \subset D$, $B_2 \subset
\Rd \setminus (D \cup \partial D)$, $\partial B_1 \cap \partial B_2 =
y$. For any $x \in D_{R/2}$ there exists a unique point $x_* \in
\partial D$ such that $\delta_D(x) = |x - x_*|$. Let $B_1 = B(z_1,R)$,
$B_2 = B(z_2,R)$ be the balls for the point $x_*$. Let $H(x)$ be the
half-space containing $B_1$ such that $\partial H(x)$ contains $x_*$ and
is perpendicular to the segment  $\overline{z_1 z_2}$.

The next proposition asserts that for small $t$, the quantity $r_D(t,x,x)$ 
can be replaced by  $r_{H(x)}(t,x,x)$. This  is a crucial
step in the proof of Theorem \ref{mainthm}. The proof is fairly 
long and technical and is deferred to 
after the proof of Theorem \ref{mainthm}.

\begin{prop}
\label{rHrD}
Let $D \subset \Rd$, $d \ge 2$, be an open bounded set with $R$-smooth boundary $\partial D$. Then for any $x \in D \setminus D_{R/2}$ and $t > 0$ such that $t^{1/\alpha} \le R/2$ we have
\begin{equation}
\label{rHrDf}
|r_D(t,x,x) - r_{H(x)}(t,x,x)|
\le \frac{c t^{1/\alpha}}{R t^{d/\alpha}}\left(\left(\frac{t^{1/\alpha}}{\delta_D(x)}\right)^{d + \alpha/2 - 1} \wedge 1\right).
\end{equation} 
\end{prop}

Let us assume the proposition and use it to 
estimate the contribution from $D\setminus D_{R/2}$ to the integral of $r_D(t, x, x)$ in (\ref{main1}).
 \bigskip
 
\noindent{\bf Claim II:} 

\begin{equation}
\label{main3}
\left| \int_{D \setminus D_{R/2}} r_D(t,x,x) \, dx - 
\int_{D \setminus D_{R/2}} r_{H(x)}(t,x,x) \, dx \right|
\le \frac{c |D| t^{2/\alpha}}{R^2 t^{d/\alpha}},
\end{equation} 
for $t^{1/\alpha} \le R/2$.
To see this observe that by  Proposition \ref{rHrD} the left hand side of (\ref{main3}) is bounded above by 
$$
\frac{c t^{1/\alpha}}{R t^{d/\alpha}}
\int_0^{R/2} |\partial D_q| \left(\left(\frac{t^{1/\alpha}}{q}\right)^{d + \alpha/2 - 1} \wedge 1\right) \, dq.
$$
By Corollary \ref{partialDq}, (i),  the last quantity is smaller than or equal to 
\begin{equation}
\label{integralt/q}
\frac{c t^{1/\alpha}|\partial D|}{R t^{d/\alpha}}
\int_0^{R/2} \left(\left(\frac{t^{1/\alpha}}{q}\right)^{d + \alpha/2 - 1} \wedge 1\right) \, dq.
\end{equation}
It is easy to show that the integral in (\ref{integralt/q}) is 
bounded above by $c t^{1/\alpha}$. Using this and 
Corollary \ref{partialDq}, (ii), we obtain (\ref{main3}). 

Recall that $H = \{(x_1,\ldots,x_d) \in \Rd: \, x_1 > 0\}$. 
For abbreviation let us denote 
$$
f_H(t,q) = r_H^{(\alpha)}(t,(q,0,\ldots,0),(q,0,\ldots,0)), \quad t,q > 0.
$$
Of course we have $r_{H(x)}(t,x,x) = f_H(t,\delta_{H(x)}(x))$.
 Note also that $f_H(t,q)$ satisfies the following properties
$$
f_H(t,q) = t^{-d/\alpha} f_H(1,q t^{-1/\alpha}), \quad
f_H(1,q) \le c (q^{-d - \alpha} \wedge 1).
$$
In the next step we will show that 
\begin{equation}
\label{main4}
\left| \int_{D \setminus D_{R/2}} r_{H(x)}(t,x,x) \, dx - 
\frac{t^{1/\alpha} |\partial D|}{ t^{d/\alpha}} \int_0^{R/(2 t^{1/\alpha})} f_H(1,q) \, dq \right|
\le \frac{c |D| t^{2/\alpha}}{R^2 t^{d/\alpha}}.
\end{equation}
Note that the constant $C_2$ which appears in the 
formulation of Theorem \ref{mainthm} satisfies $C_2 = \int_0^{\infty} f_H(1,q) \, dq$.

We have
\begin{eqnarray*}
\int_{D \setminus D_{R/2}} r_{H(x)}(t,x,x) \, dx
 &=& \int_0^{R/2} |\partial D_u|f_H(t,u) \, du  \\ 
&=&\frac{1}{ t^{d/\alpha}}  \int_0^{R/2} |\partial D_u| f_H(1,u t^{-1/\alpha}) \, du\\
&=&\frac{t^{1/\alpha}}{ t^{d/\alpha}} \int_0^{R/(2 t^{1/\alpha})} |\partial D_{t^{1/\alpha} q}| f_H(1,q) \, dq,
\end{eqnarray*}
where the second equality follows by  
scaling and the third by the substitution  $q = u t^{-1/\alpha}$.
Hence the left hand side of (\ref{main4}) is bounded  above by 
$$
\frac{t^{1/\alpha}}{ t^{d/\alpha}} \int_0^{R/(2 t^{1/\alpha})} 
\left||\partial D_{t^{1/\alpha} q}| - |\partial D| \right| f_H(1,q) \, dq.
$$
By Corollary \ref{partialDq}, (iii), this is smaller than
\begin{eqnarray*}
&& \frac{c |D| t^{2/\alpha}}{R^2 t^{d/\alpha}} 
\int_0^{R/(2 t^{1/\alpha})} q f_H(1,q) \, dq\\ 
&\le& \frac{c |D| t^{2/\alpha}}{R^2 t^{d/\alpha}} \int_0^{\infty} q (q^{-d - \alpha} \wedge 1)
 \, dq \le \frac{c |D| t^{2/\alpha}}{R^2 t^{d/\alpha}}.
\end{eqnarray*}
This gives (\ref{main4}).
Finally, we have
\begin{equation}
\label{main5}
\left| \frac{t^{1/\alpha} |\partial D|}{ t^{d/\alpha}} \int_0^{R/(2 t^{1/\alpha})} 
f_H(1,q) \, dq  - \frac{t^{1/\alpha} |\partial D|}{ t^{d/\alpha}}
 \int_0^{\infty} f_H(1,q) \, dq\right|
\le \frac{c |D| t^{2/\alpha}}{R^2 t^{d/\alpha}}.
\end{equation}
To see this recall that $R/(2 t^{1/\alpha}) \ge 1$.  So for $q \ge R/(2 t^{1/\alpha})$
 we have $f_H(1,q) \le c q^{-d-\alpha} \le c q^{-2}$. Therefore
 
$$
\int_{R/(2 t^{1/\alpha})}^{\infty} f_H(1,q) \, dq \le 
c \int_{R/(2 t^{1/\alpha})}^{\infty} \frac{dq}{q^2} \le \frac{c t^{1/\alpha}}{R}.
$$
This and Corollary \ref{partialDq}, (ii), gives (\ref{main5}).
Now, (\ref{main1}), (\ref{main2}), (\ref{main3}), 
(\ref{main4}), (\ref{main5}) give (\ref{main}).
\end{proof}

\begin{proof}[Proof of Proposition \ref{rHrD}]

Let $x_* \in \partial D$ be a unique point such that 
$|x - x_*| = \dist(x,\partial D)$ and $B_1$ and $B_2$ be 
the balls with radius $R$ such that $B_1 \subset D$, 
$B_2 \subset \Rd \setminus (D \cup \partial D)$, 
$\partial B_1 \cap \partial B_2 = x_*$. Let us also assume that $x_* = 0$ 
and choose 
an orthonormal coordinate system $(x_1,\ldots,x_d)$ 
so that the positive axis $0x_1$ is in the direction of $\vec{0p}$ 
where $p$ is the center of the ball $B_1$. Note that $x$ lies
 on the interval $0p$ so $x = (|x|,0,\ldots,0)$. Note also that
  $B_1 \subset D \subset (\overline{B_2})^c$ and 
  $B_1 \subset H(x) \subset (\overline{B_2})^c$. 
  For any open sets $A_1$, $A_2$ such that
   $A_1 \subset A_2$ we have $r_{A_1}(t,x,y) \ge r_{A_2}(t,x,y)$ so
$$
|r_D(t,x,x) - r_{H(x)}(t,x,x)| \le r_{B_1}(t,x,x) - r_{(\overline{B_2})^c}(t,x,x).
$$
Recall that for any open set $A \subset \Rd$ the function $r_A(t,x,y)$ 
satisfies the scaling property (\ref{scalingr}).
So in order to prove the proposition it suffices to show that 
\begin{eqnarray*}
&& \frac{1}{t^{d/\alpha}} 
\left( r_{B_1/{t^{1/\alpha}}}\left(1,\frac{x}{t^{1/\alpha}},\frac{x}{t^{1/\alpha}}\right)
- r_{(\overline{B_2})^c/{t^{1/\alpha}}}\left(1,\frac{x}{t^{1/\alpha}},\frac{x}{t^{1/\alpha}}\right) \right) \\\\
&\le& 
\frac{c t^{1/\alpha}}{R t^{d/\alpha}}\left(\left(\frac{t^{1/\alpha}}{\delta_D(x)}\right) \wedge 1\right),
\end{eqnarray*}
for any $x = (|x|,0,\ldots,0)$, $|x| \in (0,R/2]$.

Given the ball $B_1$, we set 
$W = B_1/{t^{1/\alpha}}$, $U = (\overline{B_2})^c/{t^{1/\alpha}} $ and $s = R/t^{1/\alpha}$. Note that 
$s$ is the radius of $W$. 
Recall that $\partial W \cap \partial U = x_* = 0$. Note also that 
$$
\frac{\delta_D(x)}{t^{1/\alpha}} = \delta_D \left(\frac{x}{t^{1/\alpha}}\right) \le 
\dist\left(\frac{x}{t^{1/\alpha}},0\right) = 
\left|\frac{x}{t^{1/\alpha}}\right|.
$$
Replacing  $x/t^{1/\alpha}$ by $x$,  it follows that in order 
to prove the proposition it suffices to show
$$
r_W(1,x,x) - r_U(1,x,x) \le c s^{-1} (|x|^{-d - \alpha/2 + 1} \wedge 1),
$$
for any $x = (|x|,0,\ldots,0)$, $|x| \in (0,s/2]$.

By Proposition \ref{pFpD} it suffices to show 
\begin{eqnarray}
\label{balls1}
&& E^x(\tau_W < 1, X(\tau_W) \in U \setminus \overline{W}; p_U(1 - \tau_W, X(\tau_W),x)) \\
\label{balls2}
&\le& c s^{-1} (|x|^{-d - \alpha/2 + 1} \wedge 1),
\end{eqnarray}
for any $x = (|x|,0,\ldots,0)$, $|x| \in (0,s/2]$.

Let us set $A = \{\tau_W < 1, X(\tau_W) \in U \setminus 
\overline{W}\}$, $f = p_U(1 - \tau_W, X(\tau_W),x)$. 
So the expression in (\ref{balls1}) is just $E^x(A;f)$. 
Let $P$ be the following set $P = B(0,s) 
\setminus (\overline{W} \cup \overline{(U^c)})$. 
We will divide $E^x(A;f)$ into 3 terms:
\begin{equation}
\label{part1}
E^x(X(\tau_W) \notin P, A;f),
\end{equation}
\begin{equation}
\label{part2}
E^x(X(\tau_W) \in P, |X(\tau_W) - x| > 1, A;f)
\end{equation}
and
\begin{equation}
\label{part3}
E^x(X(\tau_W) \in P, |X(\tau_W) - x| \le 1, A;f).
\end{equation}
We estimate each term separately.

By (\ref{ptxy}) we have
$$
p_U(1 - \tau_W, X(\tau_W),x) \le p(1 - \tau_W, X(\tau_W),x) \le 
\frac{c}{|X(\tau_W) - x|^{d + \alpha}}.
$$
Let $a$, $b$ be the centers of $W$ and $(\overline{U})^c$.  That is, set $W = B(a,s)$ and  
$(\overline{U})^c = B(b,s)$. We have
\begin{eqnarray}
\label{balls3}
&&E^x(X(\tau_W) \notin P, A;f)\\
&&\le
c E^x(X(\tau_W) \notin (B(0,s) \cup B(a,s));|X(\tau_W) - x|^{-d - \alpha})\nonumber.
\end{eqnarray}

The distribution (harmonic measure) 
$P^x(X(\tau_{B(x_0,r)}) \in \cdot)$, $x \in B(x_0,r)$
 is well known. Indeed, by  $\cite{BGR}$ we have
$$
P^x(X(\tau_{B(x_0,r)}) \in V) = C_{\alpha}^d \int_V 
\frac{(r^2 - |x - x_0|^2)^{\alpha/2} 
\, dy}{(|y - x_0|^2 - r^2)^{\alpha/2} |x- y|^d},
$$
for $x \in B(x_0,r)$ and $V \subset B^c(x_0,r)$ 
where  $C_{\alpha}^d =
 \Gamma(d/2) \pi^{-d/2-1} \sin(\pi \alpha/2)$. 
Therefore (\ref{balls3}) is bounded above by 
\begin{equation}
\label{balls4}
c \int_{B^c(0,s) \cup B^c(a,s)} \frac{(s^2 - |x - a|^2)^{\alpha/2}
 \, dy}{(|y - a|^2 - s^2)^{\alpha/2} |x- y|^{2d + \alpha}}.
\end{equation}
Note that on the set $B^c(0,s) \cup B^c(a,s)$ we have $|x - y| \ge c |a - y|$.
Changing to  polar coordinates $(\rho,\varphi_1,\ldots,\varphi_{d - 1})$ 
  centered at $a$ we see that  (\ref{balls4}) is bounded above by 
$$
c s^{\alpha} \int_{s}^{\infty} \frac{\rho^{d - 1}\, 
d\rho}{(\rho - s)^{\alpha/2} \rho^{\alpha/2} \rho^{2d + \alpha}} \le c s^{-d - \alpha}.
$$
Note that $s \ge 2$ because $t^{1/\alpha} \le R/2$. 
Using this and the fact that $|x| \in (0,s/2)$ we 
have $s^{-d - \alpha} \le  c s^{-1} (|x|^{-d - \alpha/2 + 1} \wedge 1)$. 
This  shows that $E^x(X(\tau_W) \notin P, A;f)$ is bounded by (\ref{balls2}).

Now we will estimate (\ref{part2}).
By Corollary \ref{heatcor}  we have
$$
p_U(1 - \tau_W, X(\tau_W),x)) \le \frac{c \delta_U^{\alpha/2}(X(\tau_W))}{|X(\tau_W) - x|^{d + \alpha}},
$$
on the set $|X(\tau_W) - x| > 1$. 
Thus  (\ref{part2}) is bounded  above by 
\begin{eqnarray}
\nonumber
&& c E^x(X(\tau_W) \in P; \delta_U^{\alpha/2}(X(\tau_W)) \, |X(\tau_W) - x|^{-d - \alpha}) \\
\label{balls6}
&=& c \int_{P} \frac{(s^2 - |x - a|^2)^{\alpha/2} \, \delta_U^{\alpha/2}(y) \, dy}{(|y - x_0|^2 - s^2)^{\alpha/2} |x- y|^{2d + \alpha}}.
\end{eqnarray}
Since  $(s^2 - |x - a|^2)^{\alpha/2} \le c |x|^{\alpha/2} s^{\alpha/2}$ and $(|y - x_0|^2 - s^2)^{\alpha/2} \ge c \delta_W^{\alpha/2}(y) s^{\alpha/2}$, (\ref{balls6}) is bounded  above by 
\begin{equation}
\label{balls7}
c |x|^{\alpha/2} \int_{P} \frac{\delta_U^{\alpha/2}(y) \, dy}{\delta_W^{\alpha/2}(y) \, |x- y|^{2d + \alpha}}.
\end{equation}

Let us recall that $|X(\tau_W) - x| > 1$ so $|x - y| \ge 1$ in (\ref{balls7}). Now we will use techniques developed in \cite{K}. For completeness we repeat several arguments 
 from that paper. Let us introduce spherical coordinates $y = (\rho,\varphi_1,\ldots,\varphi_{d - 1})$ 
 with the origin $0$ and principal axis $\overline{0a}$. There are small technical differences between the case $d = 2$ where $\varphi_1 \in [0,2\pi)$ and the case $d \ge 3$ where $\varphi_1 \in [0,\pi)$. 
 We will make calculations for the case $d \ge 3$. 
 The case $d = 2$ is very similar and we leave it to the reader. 

Consider the triangle $T=y0a$ with vertices  $y$, $0$, $a$. We have
$$
|y - a|^2 = |y - 0|^2 + |0 - a|^2 - 2|y - 0||0 - a| \cos\varphi_1.
$$
Since $|0 - a| = s$ and $|y - 0| = \rho$, we get
$$
|y - a|^2 = \rho^2 + s^2 - 2 \rho s \cos\varphi_1.
$$
For $0 < \rho < s$ let $\beta(\rho)$ be the angle satisfying $0 \le \beta(\rho) \le \pi/2$ and
\begin{equation}
\label{betarho}
s^2 = \rho^2 + s^2 - 2 \rho s \cos\beta(\rho).
\end{equation}
The angle $\beta(\rho)$ has the following property. $y = (\rho,\varphi_1,\ldots,\varphi_{d - 1}) \in P$ if and only if $0 < \rho < s$ and 
$$
\pi - \beta(\rho) \ge \varphi_1 \ge \beta(\rho).
$$
From (\ref{betarho}) we get 
$$
\cos\beta(\rho) = \frac{\rho}{2 s}.
$$
Thus if $y \in P$ we have $\cos\beta(\rho) < 1/2$. Hence
$$
\pi/2 \ge \beta(\rho) \ge \pi/3 \quad \text{and} \quad \sin\beta(\rho) \ge \sqrt{3}/2.
$$
Note that if $\pi/2 \ge \gamma \ge 0$ then $(\pi/2) \sin\gamma \ge \gamma$. Using this we obtain
$$
\frac{\pi \rho}{4 s} = \frac{\pi}{2} \sin\left(\frac{\pi}{2} - \beta(\rho) \right) \ge 
\frac{\pi}{2} - \beta(\rho).
$$
Hence
\begin{equation}
\label{pibeta}
\frac{\pi \rho}{2 s} \ge \pi - 2 \beta(\rho). 
\end{equation}
For $y \in P$ the double angle formula gives 
\begin{eqnarray}\label{doubleangle}
&& |y - a|^2 - s^2 = \rho^2 - 2 \rho s \cos((\varphi_1 - \beta(\rho)) + \beta(\rho))\\
&=& \rho^2 - 2 \rho s \cos\beta(\rho) \cos(\varphi_1 - \beta(\rho)) 
+ 2 \rho s \sin\beta(\rho) \sin(\varphi_1 - \beta(\rho)). \nonumber 
\end{eqnarray}
But  by (\ref{betarho}) we have $\rho^2 - 2 \rho s \cos\beta(\rho) = 0$ and this gives that 
(\ref{doubleangle}) is bounded below by 
$$
 2 \rho s \sin\beta(\rho) \sin(\varphi_1 - \beta(\rho)) 
\ge \rho s \sin(\varphi_1 - \beta(\rho)).
$$
It follows that for $y \in P$,
$$
\delta_W(y) = |y - a| - s \ge (|y - a|^2 - s^2)/s \ge \rho \sin(\varphi_1 - \beta(\rho)).
$$
Recall that $(\overline{U})^c = B(b,s)$. Similarly as above for $y \in P$ we obtain
$$
|y - b|^2 - s^2 \le \rho^2 - 2 \rho s \cos(\pi - \beta(\rho)) = 2 \rho^2,
$$
so $\delta_U(y) = |y - b| - s \le c \rho^2/s$.

We now return to  (\ref{balls7}). Let us recall that $|x - y| \ge 1$.
Let us divide $P$ into 2 sets:
$$
P_1 = \{y \in P: |y - x| \in (1, 2|x|)\},
$$
and 
$$
P_2 = \{y \in P: |y - x| \ge 1 \vee 2|x|\}.
$$
We first estimate the integral in (\ref{balls7}) over the set $P_1$. Since the set $P_1$ 
is not empty only when $|x| \ge 1/2$, we may  assume that 
$|x| \ge 1/2$. Note also that for $y \in P_1$ we have $|y - x| \ge c |x|$. It follows that
\begin{equation}
\label{P1}
c |x|^{\alpha/2} \int_{P_1} \frac{\delta_U^{\alpha/2}(y) \, dy}{\delta_W^{\alpha/2}(y) \, |x- y|^{2d + \alpha}} 
\le c |x|^{-2d - \alpha/2} \int_{P_1} \frac{\delta_U^{\alpha/2}(y) \, dy}{\delta_W^{\alpha/2}(y)}.
\end{equation}
Note that for $y \in P_1$ we have $|y| \le 3 |x|$. Using polar coordinates we obtain 
\begin{eqnarray}
\nonumber
&& \int_{P_1} \frac{\delta_U^{\alpha/2}(y) \, dy}{\delta_W^{\alpha/2}(y)}\\ 
\nonumber
&\le& c \int_0^{3 |x|} \int_{\beta(\rho)}^{\pi - \beta(\rho)} \int_0^{\pi} \ldots \int_0^{\pi} \int_0^{2 \pi} \frac{\rho^{\alpha}/s^{\alpha/2}}{\rho^{\alpha/2} \sin^{\alpha/2}(\varphi_1 - \beta{\rho})} \\ \nonumber\\
&&  \quad \quad \quad \times \rho^{d - 1} \sin^{d - 2}\varphi_1 \ldots \sin\varphi_{d - 2} \, d\varphi_{d - 1} \ldots   \, d\rho\\ \nonumber\\
\label{deltaUW}
&\le& c \int_0^{3 |x|} \rho^{d - 1} \int_{\beta(\rho)}^{\pi - \beta(\rho)} 
\frac{\rho^{\alpha}/s^{\alpha/2} \, d\varphi_1}{\rho^{\alpha/2} \sin^{\alpha/2}(\varphi_1 - \beta{\rho})} \, d\rho.
\end{eqnarray}
We now claim that 
\begin{equation}
\label{angle}
\int_{\beta(\rho)}^{\pi - \beta(\rho)} 
\frac{\rho^{\alpha}/s^{\alpha/2} \, d\varphi_1}{\rho^{\alpha/2} \sin^{\alpha/2}(\varphi_1 - \beta{\rho})} \le \frac{c \rho}{s}.
\end{equation}
Indeed, the left hand side of (\ref{angle}) equals 
$$
\frac{\rho^{\alpha/2}}{s^{\alpha/2}} \int_{0}^{\pi - 2 \beta(\rho)} 
\frac{ d\varphi}{\sin^{\alpha/2}\varphi} \le
\frac{c \rho^{\alpha/2}}{s^{\alpha/2}} \int_{0}^{\pi - 2 \beta(\rho)} 
\frac{ d\varphi}{\varphi^{\alpha/2}}.
$$
But now  (\ref{angle}) follows from (\ref{pibeta}). Hence (\ref{deltaUW}) is bounded by $c s^{-1}|x|^{d + 1}$. It follows that (\ref{P1}) is bounded by $ c s^{-1} |x|^{-d - \alpha/2 + 1}$. We have assumed that $|x| \ge 1/2$ so (\ref{P1}) is bounded by  $c s^{-1} (|x|^{-d - \alpha/2 + 1} \wedge 1)$. 

Now we will estimate (\ref{balls7}) over the set $P_2$. For $y \in P_2$ we have $|y - x| \ge c |y|$. Note also that for $y \in P_2$ we have
$$
|y| \ge |y - x| - |x| \ge (1 - |x|) \vee |x| \ge (1/2) \vee |x|.
$$
Hence,
\begin{equation}
\label{P2}
c |x|^{\alpha/2} \int_{P_2} \frac{\delta_U^{\alpha/2}(y) \, dy}{\delta_W^{\alpha/2}(y) \, |x- y|^{2d + \alpha}} 
\le c |x|^{\alpha/2} \int_{P_2} \frac{\delta_U^{\alpha/2}(y) \, dy}{\delta_W^{\alpha/2}(y) |y|^{2d + \alpha}}.
\end{equation}
Using polar coordinates this is bounded above by 
$$
c |x|^{\alpha/2} \int_{(1/2) \vee |x|}^{\infty} \rho^{-2d - \alpha} \rho^{d - 1} \int_{\beta(\rho)}^{\pi - \beta(\rho)} 
\frac{\rho^{\alpha}/s^{\alpha/2} \, d\varphi_1}{\rho^{\alpha/2} \sin^{\alpha/2}(\varphi_1 - \beta{\rho})} \, d\rho.
$$
By (\ref{angle}) this is smaller than
$$
c s^{-1} |x|^{\alpha/2} \int_{(1/2) \vee |x|}^{\infty} \rho^{-d - \alpha} \, d\rho
\le c s^{-1} (|x|^{-d - \alpha/2 + 1} \wedge 1).
$$
It follows that (\ref{part2}) is bounded by (\ref{balls2}).

Now we will estimate (\ref{part3}). For this  we may assume that $|x| \le 1$. Let $P_3 = \{y \in P: \, |y| \le 2\}$. (\ref{part3}) is bounded above by
\begin{eqnarray*}
&& E^x(X(\tau_W) \in P_3, A;f) \\
&=& E^x(\tau_W < 1/2, X(\tau_W) \in P_3; p_U(1 - \tau_W, X(\tau_W),x)) \\
&+& E^x(\tau_W \in [1/2,1], X(\tau_W) \in P_3; p_U(1 - \tau_W, X(\tau_W),x)) \\
&=& \text{I} + \text{II}.
\end{eqnarray*}
We estimate $\text{I}$ first. When $\tau_W < 1/2$ we have $1 - \tau_W > 1/2$ so by Corollary \ref{heat1cor} we obtain
$$
p_U(1 - \tau_W, X(\tau_W),x)) \le c \delta_U^{\alpha/2}(X(\tau_W)).
$$
Therefore
\begin{eqnarray}
\nonumber
\text{I} &\le&  c E^x(X(\tau_W) \in P_3; \delta_U^{\alpha/2}(X(\tau_W))) \\
\label{balls8}
&=& c \int_{P_3} \frac{(s^2 - |x - a|^2)^{\alpha/2} \, \delta_U^{\alpha/2}(y) \, dy}{(|y - x_0|^2 - s^2)^{\alpha/2} |x- y|^{d}}.
\end{eqnarray}
Using the same argument used to estimate (\ref{balls6}) by (\ref{balls7}), we obtain that (\ref{balls8}) is bounded above by
\begin{equation*}
c |x|^{\alpha/2} \int_{P_3} \frac{\delta_U^{\alpha/2}(y) \, dy}{\delta_W^{\alpha/2}(y) \, |x- y|^{d}}.
\end{equation*}
We divide $P_3$ into 2 sets:
\begin{equation}
\label{P4}
P_4 = \{y \in P: |y - x| \le 2|x|\},
\end{equation}
\begin{equation}
\label{P5}
P_5 = \{y \in P: |y - x| > 2|x|\}.
\end{equation}
As before, the  arguments used for (\ref{P1}) and (\ref{P2}) give 
$$
c |x|^{\alpha/2} \int_{P_4} \frac{\delta_U^{\alpha/2}(y) \, dy}{\delta_W^{\alpha/2}(y) \, |x- y|^{d}} \le c s^{-1} |x|^{\alpha/2 + 1},
$$
$$
c |x|^{\alpha/2} \int_{P_5} \frac{\delta_U^{\alpha/2}(y) \, dy}{\delta_W^{\alpha/2}(y) \, |x- y|^{d}} \le c s^{-1} |x|^{\alpha/2}.
$$
Using the fact that $|x| \le 1$ we finally obtain that 
$$\text{I} \le c s^{-1} (|x|^{-d - \alpha/2 + 1} \wedge 1).$$

Now we need to estimate $\text{II}$.  By the generalized space--time Ikeda-Watanabe formula (Proposition \ref{IWprop}), 
$$
\text{II} = \int_W \int_{1/2}^{1} p_W(s,x,z) \int_{P_3} \frac{\mathcal{A}_{d,-\alpha}}{|z - y|^{d + \alpha}} p_U(1-s,y,x) \, dy \, ds \, dz.
$$
We estimate $p_W(s,x,z)$ in the following way. We have $s \in [1/2,1]$. For $z \in W \cap B(0,3)$ by Corollary \ref{heat1cor} we obtain $p_W(s,x,z) \le c \delta_W^{\alpha/2}(z)$. For $z \in W \cap B^c(0,3)$ we get $p_W(s,x,z) \le p(s,x,z) \le c |x - z|^{-d - \alpha}$.
It follows that
\begin{eqnarray*}
\text{II} &=& c \int_{W \cap B(0,3)} \delta_W^{\alpha/2}(z) 
\int_{P_3} \frac{1}{|z - y|^{d + \alpha}} 
\int_{1/2}^{1} p_U(1-s,y,x) \, ds \, dy  \, dz \\
&+& c \int_{W \cap B^{c}(0,3)} \frac{1}{|x - z|^{d + \alpha}}
\int_{P_3} \frac{1}{|z - y|^{d + \alpha}} 
\int_{1/2}^{1} p_U(1-s,y,x) \, ds \, dy  \, dz.
\end{eqnarray*}
We have
$$
\int_{1/2}^{1} p_U(1-s,y,x) \, ds \le G_U(y,x),
$$
where $G_U(y,x)$ is the Green function for $U$. Hence $\text{II}$ is bounded above by
$$
c \int_{P_3} G_U(y,x) 
\left( \int_{W \cap B(0,3)} \frac{\delta_W^{\alpha/2}(z) \, dz}{|z - y|^{d + \alpha}} 
+ \int_{W \cap B^{c}(0,3)} \frac{dz}{|x - z|^{d + \alpha} |z - y|^{d + \alpha}} \right)
\, dy.
$$
For $y \in P_3$ and $z \in W$ we have $\delta_W(z) \le |z - y|$ and hence 
$$
\int_{W \cap B(0,3)} \frac{\delta_W^{\alpha/2}(z) \, dz}{|z - y|^{d + \alpha}}
\le \int_{B^{c}(y,\delta_W(y))} \frac{|z - y^{\alpha/2}| \, dz}{|z - y|^{d + \alpha}} \le
\frac{c}{\delta_W^{\alpha/2}(y)}.
$$
For $y \in P_3$ and $z \in W \cap B^{c}(0,3)$ we have $|x - z| \ge c |z|$, $|y - z| \ge c |z|$, $\delta_W^{-\alpha/2}(y) \ge c$ thus  
$$
\int_{W \cap B^{c}(0,3)} \frac{dz}{|x - z|^{d + \alpha} |z - y|^{d + \alpha}} 
\, dy \le c \int_{B^{c}(0,3)} \frac{dz}{|z|^{2d + 2 \alpha}} \le c \le \frac{c}{\delta_W^{\alpha/2}(y)}.
$$
Hence
$$
\text{II} \le c \int_{P_3} \frac{G_U(y,x)}{\delta_W^{\alpha/2}(y)} \, dy.
$$
Recall that $|x| \le 1$. By Corollary \ref{Green1} for $y \in P_3$ we get $G_U(y,x) \le c \delta_U^{\alpha/2}(y) |x - y|^{\alpha/2 - d}$. Thus  
$$
\text{II} \le c \int_{P_3} \frac{\delta_U^{\alpha/2}(y)}{|x - y|^{d - \alpha/2} \delta_W^{\alpha/2}(y)} \, dy.
$$
Finally, we can divide $P_3$ into sets $P_4$, $P_5$ (see \ref{P4}, \ref{P5}). The same arguments used for  (\ref{P1}, \ref{P2}) and the fact that $|x| \le 1$ give that  
$$\text{II} \le c s^{-1} (|x|^{-d - \alpha/2 + 1} \wedge 1).$$
This shows inequality (\ref{balls1} - \ref{balls2}) and finishes the proof.
\end{proof}

\noindent {\it Acknowledgment. We would like to thank the referee for the careful reading of the paper and the many useful comments.}

\end{document}